\documentclass[english]{amsart}
\usepackage[T1]{fontenc}
\usepackage[latin9]{inputenc}
\usepackage{mathrsfs}
\usepackage{amstext}
\usepackage{amsthm}
\usepackage{amssymb}

\makeatletter
\numberwithin{equation}{section}
\numberwithin{figure}{section}
\theoremstyle{plain}
\newtheorem{thm}{\protect\theoremname}
\theoremstyle{definition}
\newtheorem{defn}[thm]{\protect\definitionname}
\theoremstyle{plain}
\newtheorem{lem}[thm]{\protect\lemmaname}

\makeatother

\usepackage{babel}
\providecommand{\definitionname}{Definition}
\providecommand{\lemmaname}{Lemma}
\providecommand{\theoremname}{Theorem}

\begin{document}
\title{cartesian product of some combinatorially rich sets}
\author{sayan goswami}
\address{{\large{}Department of Mathematics, University of Kalyani, Kalyani-741235,
Nadia, West Bengal, India.}}
\email{{\large{}sayan92m@gmail.com}}
\keywords{{\large{}Central set, $J-$set and $C-$set.}}
\begin{abstract}
{\Large{}N. Hindman and D. Strauss had shown that, for discrete semigroups,
the cartesian product of two central sets are central. They also proved
that the product of $J$-sets and $C$-sets are also $J$-set and
$C$-set and characterized when the infinite product of these sets
are preserved. To prove these results they used the algebraic structure
of Stone-\v{C}ech compactification of discrete semigroups. In this
work we will give a combinatorial proof of the preserveness of those
large sets under finite cartesian product. }{\Large\par}
\end{abstract}

\maketitle

\section{\textbf{\Large{}introduction}}

{\Large{}A subset $S$ of $\mathbb{Z}$ is called syndetic if there
exists $r\in\mathbb{N}$ such that $\bigcup_{i=1}^{r}\left(S-i\right)=\mathbb{Z}$
and it is called thick if it contains arbitrary long intervals in
it. Sets which can be expressed as intersection of thick and syndetic
sets are called piecewise syndetic sets.}{\Large\par}

{\Large{}For a general semigroup $\left(S,\cdot\right)$, a set $A\subseteq S$
is said to be right syndetic in $\left(S,\cdot\right)$, if there
exists a finite nonempty set $F\subseteq S$ such that $\bigcup_{t\in F}t^{-1}A=S$
where $t^{-1}A=\left\{ s\in S:t\cdot s\in A\right\} $. A set $A\subseteq S$
is said to be right thick if for every finite nonempty set $E\subseteq S$,
there exists an element $x\in S$ such that $E\cdot x\subseteq A$.
A set $A\subseteq S$ is said to be right piecewise syndetic set if
there exist a finite nonempty set $F\subseteq S$ such that $\bigcup_{t\in F}t^{-1}A$
is right thick in $S$. It can be proved that a right piecewise syndetic
set is the intersection of a right thick set and a right syndetic
set.}{\Large\par}

{\Large{}There is an analogous notion of left syndetic, left thick
and left piecewise syndetic sets. For commutative semigroup this two
notions coincide. In the rest of the paper, until confusion arises,
we will drop the word 'right' and simply call a 'large' sets instead
of a 'right large' set. }{\Large\par}

{\Large{}Here we give a brief introduction of the algebraic structure
of Stone-\v{C}ech compactification of discrete semigroups $\left(S,\cdot\right)$.}{\Large\par}

{\Large{}Let $(S,\cdot)$ be a countable discrete semigroup and $\beta S$,
be the set of ultrafilters on $S$, identifying the principal ultrafilters
with the points of $S$ and thus pretending that $S\subseteq\beta S$.
Given$A\subseteq S$ let us set, 
\[
\overline{A}=\{p\in\beta S:A\in p\}.
\]
 Then the set $\{\overline{A}:A\subseteq S\}$ is a basis for a topology
on $\beta S$. The operation $\cdot$ on $S$ can be extended to the
Stone-\v{C}ech compactification $\beta S$ of $S$ so that $(\beta S,\cdot)$
is a compact right topological semigroup (meaning that for any $p\in\beta S$,
the function $\rho_{p}:\beta S\rightarrow\beta S$ defined by $\rho_{p}(q)=q\cdot p$
is continuous) with $S$ contained in its topological center (meaning
that for any $x\in S$, the function $\lambda_{x}:\beta S\rightarrow\beta S$
defined by $\lambda_{x}(q)=x\cdot q$ is continuous). Given $p,q\in\beta S$
and $A\subseteq S$, $A\in p\cdot q$ if and only if $\{x\in S:x^{-1}A\in q\}\in p$,
where $x^{-1}A=\{y\in S:x\cdot y\in A\}$. }{\Large\par}

{\Large{}A nonempty subset $I$ of a semigroup $(T,\cdot)$ is called
a left ideal of $\emph{T}$ if $T\cdot I\subset I$, a right ideal
if $I\cdot T\subset I$, and a two sided ideal (or simply an ideal)
if it is both a left and right ideal. A minimal left ideal is the
left ideal that does not contain any proper left ideal. Similarly,
we can define minimal right ideal and smallest ideal.}{\Large\par}

{\Large{}Any compact Hausdorff right topological semigroup $(T,.)$
has a smallest two sided ideal}{\Large\par}

{\Large{}
\[
\begin{array}{ccc}
K(T) & = & \bigcup\{L:L\text{ is a minimal left ideal of }T\}\\
 & = & \,\,\,\,\,\bigcup\{R:R\text{ is a minimal right ideal of }T\}
\end{array}
\]
}{\Large\par}

{\Large{}Given a minimal left ideal $L$ and a minimal right ideal
$R$, $L\cap R$ is a group, and in particular contains an idempotent.
An idempotent in $K(T)$ is called a minimal idempotent. If $p$ and
$q$ are idempotents in $T$, we write $p\leq q$ if and only if $p\cdot q=q\cdot p=p$.
An idempotent is minimal with respect to this relation if and only
if it is a member of the smallest ideal. }{\Large\par}

{\Large{}A quasi-central set is generally defined in terms of algebraic
structure of $\beta S$. }{\Large\par}
\begin{defn}
{\Large{}\label{Definition 1.1} \cite[Definition 1.2]{key-5} Let
$(S,\cdot)$ be a semigroup and let $A\subseteq S$. Then $A$ is
quasi-central if and only if there is some idempotent $p\in cl(K(\beta S))$
with $p\in\overline{A}.$}{\Large\par}
\end{defn}

{\Large{}But it has an combinatorial characterization which will be
needed for our purpose, stated below.}{\Large\par}
\begin{thm}
{\Large{}\label{Theorem 1.2} \cite[Theorem 3.8 (4)]{key-5} In an
infinite semigroup $\left(S,\cdot\right)$, $A\subseteq S$ is said
to be Quasi-central iff there is a decreasing sequence $\langle C_{n}\rangle_{n=1}^{\infty}$
of subsets of $A$ such that,}{\Large\par}
\begin{enumerate}
\item {\Large{}\label{2.1} for each $n\in\mathbb{N}$ and each $x\in C_{n}$,
there exists $m\in\mathbb{N}$ with $C_{m}\subseteq x^{-1}C_{n}$
and}{\Large\par}
\item {\Large{}\label{2.2} $C_{n}$ is piecewise syndetic $\forall\,n\in\mathbb{N}$.}{\Large\par}
\end{enumerate}
\end{thm}

{\Large{}In theorem \ref{Theorem 2.16}, we will show combinatoriallly
that product of two quasi-central set is quasi-central. }{\Large\par}

{\Large{}The definition of central set in \cite{key-5} was in terms
of topological dynamics and the definition makes sense in any semigroup.
It has a simple algebraic characterization which states that,}{\Large\par}
\begin{defn}
{\Large{}\label{definition 1.3}\cite[Definition 4.42, page 102]{key-7}
Let $(S,\cdot)$ be a semigroup and let $A\subseteq S$. Then $A$
is central if and only if there is some minimal idempotent $p\in K\left(\beta S\right)$
with $p\in\overline{A}.$}{\Large\par}
\end{defn}

{\Large{}To say the combinatorial characterization of central set
we have to know about the notion of collectionwise piecewise syndetic
sets.}{\Large\par}

{\Large{}Here we first give a combinatorial characterization generalizing
the notion of piecewise syndetic sets. }{\Large\par}

{\Large{}The notation $\mathscr{P}_{f}(S)$ means, the collection
of non-empty finite subsets of $S$.}{\Large\par}
\begin{defn}
{\Large{}\label{definition 1.4} \cite[Definition 14.19, page 353]{key-7}
Let $(S,\cdot)$ be a semigroup and let $\mathscr{A}\subseteq\mathscr{P}(S)$.
Then $\mathscr{A}$ is a $\mathit{collection}wise\,piecewise\,syndetic$
iff there exists functions $G:\mathscr{P}_{f}(\mathscr{A})\rightarrow\mathscr{P}_{f}(S)$
and $x:\mathscr{P}_{f}(\mathscr{A})\times\mathscr{P}_{f}(S)\rightarrow S$
such that for all $F\in\mathscr{P}_{f}(S)$ and all $\mathscr{F}$
and $\mathscr{H}$ in $\mathscr{P}_{f}(\mathscr{A})$ with $\mathscr{F}\subseteq$$\mathscr{H}$one
has $F\cdot x(\mathscr{F},\mathscr{H})\subseteq\bigcup_{t\in G(\mathscr{F})}t^{-1}(\bigcap\mathscr{F})$.}{\Large\par}
\end{defn}

{\Large{}We will use the following definition in our proof,}{\Large\par}
\begin{defn}
{\Large{}\label{definition 1.5} \cite[Page 63]{key11} A family $\mathcal{I}$
of subsets of $S$ is said to be collectionwise piecewise syndetic
if for each $A\in\mathcal{I}$, there exists a finite set $K_{A}\subset S$
such that $\left\{ K_{A}^{-1}A:A\in\mathcal{I}\right\} $ is collectionwise
thick, i.e, intersection of any finite sub-family of $\left\{ K_{A}^{-1}A:A\in\mathcal{I}\right\} $
is thick in $S$.}{\Large\par}
\end{defn}

{\Large{}So every member of a family which is $\mathit{collection}wise\,piecewise\,syndetic$,
is $piecewise\,syndetic$. The following one is the combinatorial
characterization of central set which we will use in future.}{\Large\par}
\begin{thm}
{\Large{}\label{Theorem 1.6} \cite[Theorem 3.8 (5)]{key-7} For a
countable semigroup $\left(S,.\right)$, $A\subseteq S$ is said to
be central iff there is a decreasing sequence $\langle C_{n}\rangle_{n=1}^{\infty}$
of subsets of $A$ such that,}{\Large\par}
\begin{enumerate}
\item {\Large{}\label{6.1} for each $n\in\mathbb{N}$ and each $x\in C_{n}$,
there exists $m\in\mathbb{N}$ with $C_{m}\subseteq x^{-1}C_{n}$
and}{\Large\par}
\item {\Large{}\label{6.2} $C_{n}$ is collectionwise piecewise syndetic
$\forall\,n\in\mathbb{N}$.}{\Large\par}
\end{enumerate}
\end{thm}

{\Large{}In theorem \ref{Theorem 2.17}, we will show combinatoriallly
that product of two central set is central. }{\Large\par}

{\Large{}The following famous central set theorem is due to H. Furstenberg.}{\Large\par}
\begin{thm}
{\Large{}\label{ (Furstenberg-Central Sets Theorem )} \cite{key-3}
Let $A$ be a central subset of $\mathbb{N}$, let $k\in\mathbb{N}$
and for each $i\in\left\{ 1,2,...,k\right\} $, let $\left\langle y_{l,n}\right\rangle _{n=1}^{\infty}$
be a sequence in $\mathbb{Z}$. There exist sequences $\left\langle a_{n}\right\rangle _{n=1}^{\infty}$
in $\mathbb{N}$ and $\left\langle H_{n}\right\rangle _{n=1}^{\infty}$
in $P_{f}\left(\mathbb{N}\right)$ such that }{\Large\par}

{\Large{}(1) for each $n,\:\max H_{n}<\min H_{n+1}$ and }{\Large\par}

{\Large{}(2) for each $i\in\left\{ 1,2,...,k\right\} $ and each $F\in P_{f}\left(\mathbb{N}\right)$,
\[
\sum_{n\in F}\left(a_{n}+\sum_{t\in H_{n}}y_{i,t}\right)\in A.
\]
}{\Large\par}
\end{thm}

{\Large{}The following theorem is called central set theorem for commutative
semigroup.}{\Large\par}
\begin{thm}
{\Large{}\label{Theorem 1.8} \cite[Theorem 14.11, page 340]{key-7}
Let $\left(S,+\right)$ be a countable commutative semigroup, let
$A$ be a central set in $S$, and for each $l\in\mathbb{N}$, let
$\left\langle y_{l,n}\right\rangle _{n=1}^{\infty}$ be a sequence
in $S$. There exist a sequence $\left\langle a_{n}\right\rangle _{n=1}^{\infty}$
in $S$ and a sequence $\left\langle H_{n}\right\rangle _{n=1}^{\infty}$
in $\mathcal{P}_{f}\left(\mathbb{N}\right)$ such that $\max H_{n}<\min H_{n+1}$
for each $n\in\mathbb{N}$ and such that for each $f\in\Phi$, 
\[
FS\left(\left\langle a_{n}+\sum_{t\in H_{n}}y_{f\left(n\right),t}\right\rangle _{n=1}^{\infty}\right)\subseteq A,
\]
 where $\Phi$ is the set of all functions $f:\mathbb{N\rightarrow N}$
for which $f\left(n\right)\leq n$ for all $n\in\mathfrak{\mathbb{N}}$}.
\end{thm}

{\Large{}The following central set theorem is the strengthening of
the above central set theorem for general commutative semigroup:}{\Large\par}
\begin{thm}
{\Large{}\label{( Stronger Central Sets Theorem )}\cite[Theorem 14.8.4, page 337]{key-7}
Let $\left(S,+\right)$ be a commutative semigroup. Let $C$ be a
central subset of $S$. Then there exist functions $\alpha:\mathcal{P}_{f}\left(S^{\mathbb{N}}\right)\rightarrow\mathbb{N}$
such that }{\Large\par}

{\Large{}1) let $F,G\in\mathcal{P}_{f}\left(S^{\mathbb{N}}\right)$
and $F\subsetneq G$, then $\max H\left(F\right)<\min H\left(G\right)$,}{\Large\par}

{\Large{}2) whenever $r\in\mathbb{N}$, $G_{1},G_{2},...,G_{r}\in\mathcal{P}_{f}\left(^{\mathbb{N}}S\right)$
such that $G_{1}\subsetneq G_{2}\subsetneq.....\subsetneq G_{r}$
and for each $i\in\left\{ 1,2,....,r\right\} $, $f_{i}\in G_{i}$
one has 
\[
\sum_{i=1}^{r}\left(\alpha\left(G_{i}\right)+\sum_{t\in H\left(G_{i}\right)}f_{i}\left(t\right)\right)\in C.
\]
}{\Large\par}
\end{thm}

{\Large{}Various stronger non-commutative version of the above central
set theorem can be found in \cite{key-2}.}{\Large\par}

{\Large{}There is an important set, which is intimately related to
central set theorem is known as $J$-set. In commutative semigroup
it is defined as,}{\Large\par}
\begin{defn}
{\Large{}\label{definition 1.10} \cite{key-7} Let $(S,+)$ be a
commutative semigroup and let $A\subseteq S$ is said to be a $J$-set
iff whenever $F\in\mathcal{P}_{f}\left(S^{\mathbb{N}}\right)$, there
exist $a\in S$ and $H\in\mathcal{P}_{f}(\mathbb{N})$ such that for
each $f\in F,$ $a+\sum_{t\in H}f(t)\in A$.}{\Large\par}
\end{defn}

{\Large{}In non-commutative case the situation is little different.
Here the analogous notion of $J$-sets are defined as,}{\Large\par}
\begin{defn}
{\Large{}\label{definition 1.11}\cite[Definition 14.14.1, page 342]{key-7}
Let $\left(S,\cdot\right)$ is a semigroup.}{\Large\par}
\begin{enumerate}
\item {\Large{}$\mathcal{T}=S^{\mathbb{N}}$}{\Large\par}
\item {\Large{}For $m\in\mathbb{N}$, $\mathcal{J}_{m}=$$\left\{ \begin{array}{cc}
\left(t\left(1\right),t\left(2\right),\ldots,t\left(m\right)\right)\in\mathbb{N}^{m} & :\\
t\left(1\right)<t\left(2\right)<\ldots<t\left(m\right)
\end{array}\right\} $}{\Large\par}
\item {\Large{}Given $m\in\mathbb{N}$, $a\in S^{m+1}$, $t\in\mathcal{J}_{m}$
and $f\in\mathcal{T}$, 
\[
x\left(m,a,t,f\right)=\left(\prod_{j=1}^{m}\left(a\left(j\right)\cdot f\left(t\left(j\right)\right)\right)\right)\cdot a\left(m+1\right)
\]
}{\Large\par}
\item {\Large{}$A\subseteq S$ is called a $J-set$ iff for each $F\in\mathcal{P}_{f}\left(\mathcal{T}\right)$,
there exists $m\in\mathbb{N}$, $a\in S^{m+1}$, $t\in\mathcal{J}_{m}$
such that, for each $f\in\mathcal{T}$,}\\
{\Large{}
\[
x\left(m,a,t,f\right)\in A.
\]
}{\Large\par}
\end{enumerate}
{\Large{}It can be shown that the set $J\left(S\right)=\left\{ p\in\beta S:for\,all\,A\in p,A\,is\,a\,J-set\right\} $
is a compact two sided ideal of $\left(\beta S,\cdot\right)$ and
hence from \cite[Theorem 2.5, page 40]{key-7} there exists idempotents
in $J\left(S\right)$. The members of the idempotens are called $C$-sets.
For details one can see \cite[chapter 14]{key-7}.}{\Large\par}

{\Large{}In \cite{key-6}, it was shown that product of two $J$-sets
and $C$-sets are $J$-set and $C$-set. We will give a combinatorial
proof of these results without using of algebra of Stone-\v{C}ech
compactification. To give a combinatorial proof, we need the following
combinatorial characterization of $C$set in terms of $J$-sets,}{\Large\par}
\end{defn}

\begin{thm}
{\Large{}\label{Theorem 1.12} \cite[Theorem 14.27, page 358]{key-7}
For a countable semigroup $(S,\cdot)$, $A\subseteq S$ is a $C$-set
iff there is a decreasing sequence $\langle C_{n}\rangle_{n=1}^{\infty}$
of subsets of $A$ such that,}{\Large\par}

{\Large{}$(1)$ \label{property 12.1} for each $n\in\mathbb{N}$
and each $x\in C_{n}$, there exists $m\in\mathbb{N}$ with $C_{m}\subseteq x^{-1}C_{n}$
and}{\Large\par}

{\Large{}$(2)$ \label{property 12.2} $C_{n}$ is a $J$set $\forall n\in\mathbb{N}$.}{\Large\par}
\end{thm}

{\Large{}For various other notion of large sets and their property
one can see \cite{key-4} and \cite{key-8}.}{\Large\par}

\textbf{\Large{}Orientation of the paper:}{\Large{} In section 2,
we will prove the product of quasi-central and central sets are quasi-central
and central. In section 3, we first prove that product of two $J$-sets
and $C$-sets are $J$-set and $C$-set.}{\Large\par}

\section{\textbf{\Large{}product of central sets }}
\begin{lem}
{\Large{}\label{Lemma 2.13} Let $\left(S_{1},\cdot\right)$ and $\left(S_{2},\cdot\right)$
be two semigroups. Let, $A\subseteq S_{1}$ and $B\subseteq S_{2}$
are two syndetic sets then $A\times B$ is syndetic in $S_{1}\times S_{2}$.}{\Large\par}
\end{lem}

\begin{proof}
{\Large{}Let $F_{1}$ and $F_{2}$ be two finite subsets of $S_{1}$
and $S_{2}$ such that $\bigcup_{t\in F_{1}}t^{-1}A=S_{1}$ and $\bigcup_{t\in F_{2}}t^{-1}A=S_{2}$.
Now, it would be easy to verify that $\bigcup_{\left(t,s\right)\in F_{1}\times F_{2}}\left(t,s\right)^{-1}\left(A\times B\right)=S_{1}\times S_{2}$,
showing the syndeticity of $A\times B$. So we left the rest of the
verification to the reader.}{\Large\par}
\end{proof}
\begin{lem}
{\Large{}\label{Lemma 2.14} Let $\left(S_{1},\cdot\right)$ and $\left(S_{2},\cdot\right)$
be two semigroups. Let $A\subseteq S_{1}$ and $B\subseteq S_{2}$
are two thick sets then $A\times B$ is thick in $S_{1}\times S_{2}$.}{\Large\par}
\end{lem}

\begin{proof}
{\Large{}Take any finite set $F$ in $S_{1}\times S_{2}$ and take
the two sets $\pi_{1}(F)$ and $\pi_{2}(F)$, the projection of $F$
in $S_{1}$ and $S_{2}$ respectively. Now as $A$ and $B$ are thick
sets, there exists $x\in S_{1}$and $y\in S_{2}$ such that $\pi_{1}(F)\cdot x\subset A$
and $\pi_{1}(F)\cdot y\subset B$ and so, $F\cdot(x,y)\subset A\times B$.}{\Large\par}
\end{proof}
\begin{thm}
{\Large{}\label{Theorem 2.15} Let $\left(S_{1},\cdot\right)$ and
$\left(S_{2},\cdot\right)$ be two semigroups. Let $A\subseteq S_{1}$
and $B\subseteq S_{2}$ are two piecewise syndetic sets then $A\times B$
is piecewise syndetic in $S_{1}\times S_{2}$.}{\Large\par}
\end{thm}

\begin{proof}
{\Large{}Let $F$ and $G$ be two finite subsets of $S_{1}$ and $S_{2}$
respectively such that, $\bigcup_{x\in F}x^{-1}A$ and $\bigcup_{y\in G}y^{-1}B$
are thick sets in $S_{1}$ and $S_{2}.$}{\Large\par}

{\Large{}So, from lemma \ref{Lemma 2.14} the product set, $\bigcup_{(x,y)\in F\times G}(x,y)^{-1}(A\times B)$
is thick set. So, $A\times B$ is piecewise syndetic in $S_{1}\times S_{2}$.}{\Large\par}
\end{proof}
\begin{thm}
{\Large{}\label{Theorem 2.16} Let $\left(S_{1},\cdot\right)$ and
$\left(S_{2},\cdot\right)$ be two semigroups. Let $A\subseteq S_{1}$
and $B\subseteq S_{2}$ are two quasi-central sets then $A\times B$
is quasi-central in $S_{1}\times S_{2}$.}{\Large\par}
\end{thm}

\begin{proof}
{\Large{}Consider the following two chains, with property \ref{2.1}
and \ref{2.2} guaranteed by theorem\ref{Theorem 1.2}.}{\Large\par}

{\Large{}
\[
A\supseteq A_{1}\supseteq A_{2}\supseteq\ldots\supseteq A_{n}\supseteq\ldots
\]
}{\Large\par}

{\Large{}
\[
B\supseteq B_{1}\supseteq B_{2}\supseteq\ldots\supseteq B_{n}\supseteq\ldots
\]
Now consider the following chain,
\[
A\times B\supseteq A_{1}\times B_{1}\supseteq A_{2}\times B_{2}\supseteq\ldots\supseteq A_{n}\times B_{n}\supseteq\ldots
\]
We claim that the above chain satisfies the conditions of theorem
\ref{Theorem 1.2}.}{\Large\par}

{\Large{}From theorem \ref{Theorem 2.15} each $A_{i}\times B_{i}$
is piecewise syndetic in $S_{1}\times S_{2}$ and so, property \ref{2.2}
of theorem \ref{Theorem 1.2} has been satisfied.}{\Large\par}

{\Large{}Now choose $n\in\mathbb{N}$ and $(x,y)\in A_{n}\times B_{n}$,
then we claim that, there exists $m\in\mathbb{N}$ such that,}{\Large\par}

{\Large{}
\[
(x,y)^{-1}(A_{n}\times B_{n})\supseteq A_{m}\times B_{m}.
\]
}{\Large\par}

{\Large{}Now, $x\in A_{n}$ implies there exist $m_{1}\in\mathbb{N}$
such that $x^{-1}A_{n}\supseteq A_{m_{1}}$ and $y\in B_{n}$ implies
there exist $m_{2}\in\mathbb{N}$ such that $x^{-1}B_{n}\supseteq B_{m_{2}}$.}{\Large\par}

{\Large{}Take, $m=\max\{m_{1},m_{2}\}$ and so, $x^{-1}A_{n}\supseteq A_{m}$,
$y^{-1}B_{n}\supseteq B_{m}$.}{\Large\par}

{\Large{}Hence, $(x,y)^{-1}(A_{n}\times B_{n})\supseteq A_{m}\times B_{m}$
implying the property \ref{2.2} of theorem \ref{Theorem 1.2}. }{\Large\par}

{\Large{}So, $A\times B$ is quasi-central in $S_{1}\times S_{2}$.}{\Large\par}
\end{proof}
\begin{thm}
{\Large{}\label{Theorem 2.17} Let $\left(S_{1},\cdot\right)$ and
$\left(S_{2},\cdot\right)$ be two commutative semigroups. Let $A\subseteq S_{1}$
and $B\subseteq S_{2}$ are two central sets then $A\times B$ is
central in $S_{1}\times S_{2}$.}{\Large\par}
\end{thm}

\begin{proof}
{\Large{}Consider the following two chains guaranteed by theorem \ref{Theorem 1.6},}{\Large\par}

{\Large{}
\[
A\supseteq A_{1}\supseteq A_{2}\supseteq\ldots\supseteq A_{n}\supseteq\ldots
\]
}{\Large\par}

{\Large{}
\[
B\supseteq B_{1}\supseteq B_{2}\supseteq\ldots\supseteq B_{n}\supseteq\ldots
\]
Then, from theorem \ref{Theorem 2.16}, the following chain satisfies
the properties of theorem \ref{Theorem 1.2}.
\[
A\times B\supseteq A_{1}\times B_{1}\supseteq A_{2}\times B_{2}\supseteq\ldots\supseteq A_{n}\times B_{n}\supseteq\ldots
\]
It remains to show that $\left\{ A_{i}\times B_{i}:i\in\mathbb{N}\right\} $
is collectionwise piecewise syndetic.}{\Large\par}

{\Large{}As, $\left\{ A_{i}:i\in\mathbb{N}\right\} $ and $\left\{ B_{i}:i\in\mathbb{N}\right\} $
are collectionwise piecewise syndetic in $S_{1}$ and $S_{2}$, we
have collection of finite sets $\left\{ K_{A_{i}}:i\in\mathbb{N}\right\} $
in $S_{1}$ and $\left\{ K_{B_{i}}:i\in\mathbb{N}\right\} $ in $S_{2}$
such that $\left\{ K_{A_{i}}^{-1}A_{i}:i\in\mathbb{N}\right\} $ and
$\left\{ K_{B_{i}}^{-1}B_{i}:i\in\mathbb{N}\right\} $ collectionwise
thick in $S_{1}$ and $S_{2}$.}{\Large\par}

{\Large{}We claim that, $\left\{ (K_{A_{i}}\times K_{B_{i}})^{-1}\left(A_{i}\times B_{i}\right):i\in\mathbb{N}\right\} $
is collectionwise thick in $S_{1}\times S_{2}$.}{\Large\par}

{\Large{}To prove the claim, take any finite set $F\subset\mathbb{N}$
and $F=\left\{ i_{1},i_{2},\ldots,i_{n}\right\} $ and consider the
set $\left\{ (K_{A_{i}}\times K_{B_{i}})^{-1}\left(A_{i}\times B_{i}\right):i\in F\right\} $.
We will show, 
\[
\bigcap_{i\in F}(K_{A_{i}}\times K_{B_{i}})^{-1}\left(A_{i}\times B_{i}\right)
\]
is thick.}{\Large\par}

{\Large{}Choose any finite set $\left\{ (a_{1},b_{1}),(a_{2},b_{2}),\ldots,(a_{n},b_{n})\right\} $.
Now, $\bigcap_{i\in F}K_{A_{i}}^{-1}A_{i}$ and $\bigcap_{i\in F}K_{B_{i}}^{-1}B_{i}$
are thick in $S_{1}$ and $S_{2}$ respectively.}{\Large\par}

{\Large{}So, there exists $x\in S_{1}$ and $y\in S_{2}$ such that
$\left\{ a_{1},a_{2},\ldots,a_{n}\right\} \cdot x\subset\bigcap_{i\in F}K_{A_{i}}^{-1}A_{i}$
and $\left\{ b_{1},b_{2},\ldots,b_{n}\right\} \cdot y\subset\bigcap_{i\in F}K_{B_{i}}^{-1}B_{i}$
and so, 
\[
\left\{ (a_{1},b_{1}),(a_{2},b_{2}),\ldots,(a_{n},b_{n})\right\} \cdot(x,y)\subset\bigcap_{i\in F}(K_{A_{i}}\times K_{B_{i}})^{-1}\left(A_{i}\times B_{i}\right)
\]
This proves collectionwise thickness of $\left\{ (K_{A_{i}}\times K_{B_{i}})^{-1}\left(A_{i}\times B_{i}\right):i\in\mathbb{N}\right\} $
and so, $\left\{ A_{i}\times B_{i}:i\in\mathbb{N}\right\} $ is collectionwise
piecewise syndetic.}{\Large\par}

{\Large{}This shows $A\times B$ is central in $S_{1}\times S_{2}$.}{\Large\par}
\end{proof}

\section{\textbf{\Large{}product of $C$-sets}}

{\Large{}In this section we first prove that, the product of two $J$-sets
are $J$-set and product of two }\textbf{\Large{}$C$-}{\Large{}sets}\textbf{\Large{}
}{\Large{}are }\textbf{\Large{}$C$}{\Large{}-set}\textbf{\Large{}
}{\Large{}in commutative semigroup and then it will be extended to
general semigroup. But to proceed further, we need the some basic
facts of partial semigroup.}{\Large\par}
\begin{defn}
{\Large{}\label{Defn 3.18} A partial semigroup is defined as a pair
$(S,\ast)$ where $\ast$ maps a subset of $S\times S$ to $S$ and
satisfies for all $a,b,c\in S$ , $(a\ast b)\ast c=a\ast(b\ast c)$
in the sense that if either side is defined, then so is the other
and they are equal.}{\Large\par}
\end{defn}

{\Large{}$\left(\mathcal{P}_{f}\left(\mathbb{N}\right),\uplus\right)$,
where $\mathcal{P}_{f}\left(\mathbb{N}\right)$ is the collection
of nonempty finite subsets of $\mathbb{N}$, is an example of partial
semigroup. Where $A\uplus B=A\cup B$ is defined iff $A\cap B=\emptyset$.}{\Large\par}

{\Large{}There are several notion of largeness of sets in partial
semigroup similar to semigroup arises from similar characterization
of the structure of Stone-\v{C}ech compactification of semigroup,
$\beta S$. The Stone-\v{C}ech compactification of a partial semigroup
contains an interesting subsemigroup, which gives us many important
notions of largeness of sets. }{\Large\par}
\begin{defn}
{\Large{}\label{Defn 3.19} Let $\left(S,\ast\right)$ be a partial
semigroup.}{\Large\par}

{\Large{}(a) For $s\in S$, $\varphi\left(s\right)=\left\{ t\in S:\:s\ast t\text{ is defined}\right\} $.}{\Large\par}

{\Large{}(b) For $H\in\mathcal{P}_{f}\left(S\right),\:\sigma\left(H\right)=\bigcap_{s\in H}\varphi\left(s\right)$.}{\Large\par}

{\Large{}(c) $\left(S,\ast\right)$ is adequate iff $\sigma\left(H\right)\neq\emptyset$
for all $H\in\mathcal{P}_{f}\left(S\right)$.}{\Large\par}

{\Large{}(d) $\delta S=\bigcap_{x\in S}\overline{\varphi\left(x\right)}=\bigcap_{H\in\mathcal{P}_{f}\left(S\right)}\overline{\sigma\left(H\right)}$.}{\Large\par}
\end{defn}

{\Large{}In fact, for a semigroup $S$, $\delta S=\beta S$ and adequacy
is needed in partial semigroup to guarantee that $\delta S\neq\emptyset$.
$(\delta S,+)$ is a compact right topological semigroup . In particular,
the combinatorially and algebraically defined large sets are sometimes
different for partial semigroups. Although, they coincide for semigroups.
In partial semigroup, each algebraically defined large sets have a
combinatorial version which is defined as using a notions precede
by \v{c}. For details on partial semigroup, one can see \cite{key-10}.}{\Large\par}
\begin{defn}
{\Large{}\label{Defn 3.20} Let $\left(S,\ast\right)$ be an adequate
partial semigroup and suppose $A\subseteq S$.}{\Large\par}

{\Large{}(a) $A$ is $IP$ if and only if there exists an idempotent
$p\in\delta S$ such that $A\in p$.}{\Large\par}

{\Large{}(b) $A$ is \v{c}-$ip$ (combinatorially $ip$) if and only
if there exists a sequence $\left\langle x_{n}\right\rangle _{n=1}^{\infty}$
in $S$ such that for all $F\in\mathcal{P}_{f}\left(\mathbb{N}\right)$,}\\
{\Large{} $\prod_{n\in F}x_{n}$ is defined and $\prod_{n\in F}x_{n}\in A$.}{\Large\par}
\end{defn}

{\Large{}A set is said to be $IP^{*}$ if it intersects every $IP$
set, i.e, it belongs to every idempotent $p\in\delta S$ and is said
to be \v{c}-$IP^{*}$ if the set intersects every \v{c}-$IP$ set.}{\Large\par}

\subsubsection{\textbf{\Large{}Commutative case:}}
\begin{thm}
{\Large{}\label{Theorem 3.21} Let $\left(S,\cdot\right)$ be a commutative
semigroup, $A\subseteq S$ is a $J$-set then, for any $F\in\mathcal{P}_{f}\left(S^{\mathbb{N}}\right)$
, 
\[
\left\{ H\in\mathcal{P}_{f}(\mathbb{N}):for\,some\,a\in S,\,a+\sum_{t\in H}f(t)\in A\,\forall f\in F\right\} 
\]
is \v{c}-ip$^{\star}$ in $\mathcal{P}_{f}(\mathbb{N})$.}{\Large\par}
\end{thm}

\begin{proof}
{\Large{}Let us choose any \v{c}-ip set in $\mathcal{P}_{f}(\mathbb{N})$
arbitrarily and let,}\\
{\Large{} $\left\{ H_{1}<H_{2}<\cdots<H_{n}<\cdots\right\} $ be the
base of that \v{c}-ip set.}{\Large\par}

{\Large{}Consider the set $G\in\mathcal{P}_{f}\left(S^{\mathbb{N}}\right)$
as,}\\
{\Large{} 
\[
\left\{ \left(\prod_{t\in H_{1}}f(t),\prod_{t\in H_{2}}f(t),\cdots,\prod_{t\in H_{n}}f(t),\cdots\right):f\in F\right\} .
\]
}{\Large\par}

{\Large{}Then, as $A$ is $J$-set, there exists $a\in S$ and $K\in\mathcal{P}_{f}(\mathbb{N})$
such that, 
\[
\left\{ a\cdot\prod_{\bigcup_{j\in K}H_{j}}f(t):t\in\bigcup_{j\in K}H_{j}\,and\,f\in F\right\} \subset A
\]
and so, 
\[
\bigcup_{j\in K}H_{j}\in\left\{ H\in\mathcal{P}_{f}(\mathbb{N}):a\cdot\prod_{t\in H}f(t)\in A\,\forall f\in F\right\} 
\]
showing that $H$ is a \v{c}-ip$^{\star}$ in $\mathcal{P}_{f}(\mathbb{N})$.}{\Large\par}
\end{proof}
{\Large{}To proceed further, we need the following equivalencies of
Hindman's theorem \cite{key-9}:}{\Large\par}
\begin{thm}
{\Large{}\label{Theiorem 3.22} The following statements are equivalent: }{\Large\par}
\begin{enumerate}
\item {\Large{}Let $\left(S,\cdot\right)$ be a semigroup, let $r\in\mathbb{N}$,
and let $\langle x_{n}\rangle_{n=1}^{\infty}$ be a sequence in S.
Let, $FP\left(\langle x_{n}\rangle_{n=1}^{\infty}\right)=\bigcup_{i=1}^{r}A_{i}$,
there exists $i\in\left\{ 1,2,\ldots,r\right\} $ and a sequence $\langle y_{n}\rangle_{n=1}^{\infty}$
in $S$ such that $FP\left(\langle y_{n}\rangle_{n=1}^{\infty}\right)\subseteq A_{i}$.}{\Large\par}
\item {\Large{}Let $r\in\mathbb{N}$ and $\mathbb{N}=\bigcup_{i=1}^{r}A_{i}$,
then there exists $i\in\left\{ 1,2,\ldots,r\right\} $ and a sequence
$\langle x_{n}\rangle_{n=1}^{\infty}$ in $S$ such that $FP\left(\langle x_{n}\rangle_{n=1}^{\infty}\right)\subseteq A_{i}$.}{\Large\par}
\item {\Large{}Let $\mathcal{P}_{f}\left(\mathbb{N}\right)$ be the set
of all finite subsets of $\mathbb{N}$ and $r\in\mathbb{N}$. If $\mathcal{P}_{f}\left(\mathbb{N}\right)=\bigcup_{i=1}^{r}\mathscr{A}_{i}$,
then there exists $i\in\left\{ 1,2,\ldots,r\right\} $ and a sequence
$\langle\mathcal{F}_{n}\rangle_{n=1}^{\infty}$ such that $\min\mathcal{F}_{n+1}>\max\mathcal{F}_{n}$
for each $n\in\mathbb{N}$ and $\bigcup_{n\in G}\mathcal{F}_{n}\in\mathcal{\mathscr{A}}_{i}$,
whenever $G\in\mathcal{F}.$}{\Large\par}
\item {\Large{}Let, $r\in\mathbb{N}$ and $\langle\mathcal{F}_{n}\rangle_{n=1}^{\infty}\subseteq\mathcal{P}_{f}\left(\mathbb{N}\right)$
be a sequence such that $\min\mathcal{F}_{n+1}>\max\mathcal{F}_{n}$
for each $n\in\mathbb{N}$ and $\bigcup_{n\in G}\mathcal{F}_{n}\in\bigcup_{i=1}^{r}\mathcal{\mathscr{A}}_{i}$,
then there exists $i\in\left\{ 1,2,\ldots,r\right\} $ and a sequence
$\langle\mathcal{G}_{n}\rangle_{n=1}^{\infty}$ such that $\min\mathcal{G}_{n+1}>\max\mathcal{G}_{n}$
for each $n\in\mathbb{N}$ and $\bigcup_{n\in H}\mathcal{G}_{n}\in\mathcal{\mathscr{A}}_{i}$,
whenever $H\in\mathcal{F}.$}{\Large\par}
\end{enumerate}
\end{thm}

\begin{proof}
{\Large{}The equivallence of $\left(1\right)$,$\left(2\right)$ and
$\left(3\right)$ can be found in \cite{key-1} and the equivallence
of $\left(2\right)$ and $\left(4\right)$ can be found in \cite[section 2.2, page 44]{key11}.}{\Large\par}
\end{proof}
{\Large{}The combinatorial proof of the following lemma uses the above
theorem which is an important,}{\Large\par}
\begin{lem}
{\Large{}\label{Lemma 3.23} In $\left(\mathcal{P}_{f}\left(\mathbb{N}\right),\uplus\right)$,
Two \v{c}-ip$^{\star}$ sets always intersect.}{\Large\par}
\end{lem}

\begin{proof}
{\Large{}Let $A$ and $B$ be two \v{c}-ip$^{\star}$ sets. Let us
take an \v{c}-ip set $FU\left(\langle\mathcal{F}_{n}\rangle_{n=1}^{\infty}\right)$
($FU$ stands for finite union). Now consider the partition $FU\left(\langle\mathcal{F}_{n}\rangle_{n=1}^{\infty}\right)=A\bigcup A'$
($A'$ is the complement of $A$), then one of $A$ or $A'$ will
contain a subsequence $\langle\mathcal{G}_{n}\rangle_{n=1}^{\infty}$
of $\langle\mathcal{F}_{n}\rangle_{n=1}^{\infty}$ such that $FU\left(\langle\mathcal{G}_{n}\rangle_{n=1}^{\infty}\right)$
is entirely contained in that set. As, $A$ is \v{c}-ip$^{\star}$,
$FU\left(\langle\mathcal{G}_{n}\rangle_{n=1}^{\infty}\right)$ can't
be contained in $A'$ and so $FU\left(\langle\mathcal{G}_{n}\rangle_{n=1}^{\infty}\right)$$\subseteq A$
and as $B$ is \v{c}-ip$^{\star}$, $B$ will intersect $FU\left(\langle\mathcal{G}_{n}\rangle_{n=1}^{\infty}\right)$
and so, $A\bigcap B\neq\emptyset$.}{\Large\par}
\end{proof}
\begin{thm}
{\Large{}\label{Theorem 3.24} Let $\left(S_{1},\cdot\right)$ and
$\left(S_{2},\cdot\right)$ be two commutative semigroups, $A\subseteq S_{1}$
and $B\subseteq S_{2}$ are two $J$-sets then $A\times B$ is $J$-set
in $S_{1}\times S_{2}$.}{\Large\par}
\end{thm}

\begin{proof}
{\Large{}Consider any $F\in\mathcal{P}_{f}\left(\left(S_{1}\times S_{2}\right)^{\mathbb{N}}\right)$,
and then}\\
{\Large{} $F=\left\{ (f^{i},g^{i}):1\leq i\leq k\,for\,some\,k\right\} $
where,}\\
{\Large{} $f^{i}\in\mathcal{P}_{f}\left(S_{1}^{\mathbb{N}}\right),g^{i}\in\mathcal{P}_{f}\left(S_{2}^{\mathbb{N}}\right)$
for all $1\leq i\leq k$.}{\Large\par}

{\Large{}Now take $F_{1}=\left\{ f^{i}:1\leq i\leq k\right\} $ and
$F_{2}=\left\{ g^{i}:1\leq i\leq k\right\} $. }\\
{\Large{}As, $A\subseteq S_{1}$ and $B\subseteq S_{2}$ are $J-sets$,
there exists two \v{c}-ip$^{\star}$ sets $K_{1}$ and $K_{2}$ in
$\mathcal{P}_{f}(\mathbb{N})$ guaranteed by theorem \ref{Theorem 3.21}.
and so from lemma \ref{Lemma 3.23} $\exists H\in K_{1}\cap K_{2}$.}{\Large\par}

{\Large{}From theorem \ref{Theorem 3.21}, there exists $a\in S_{1}$
and $b\in S_{2}$ such that $a\cdot\prod_{t\in H}f^{i}(t)\in A$ and
$a\cdot\prod_{t\in H}g^{i}(t)\in A$ for all $f^{i}\in F_{1}$ and
$g^{i}\in F_{2}$.}{\Large\par}

{\Large{}Hence, $\left(a,b\right)\cdot\prod_{t\in H}(f^{i},g^{i})(t)\in A\times B$,
proving it is a $J-set$.}{\Large\par}
\end{proof}
\begin{thm}
{\Large{}\label{Theorem 3.25} Let $\left(S_{1},\cdot\right)$ and
$\left(S_{2},\cdot\right)$ be two commutative semigroups, $A\subseteq S_{1}$
and $B\subseteq S_{2}$ are two $C$-sets then $A\times B$ is $C$-set
in $S_{1}\times S_{2}$.}{\Large\par}
\end{thm}

\begin{proof}
{\Large{}As, from theorem \ref{Theorem 1.12} $A$ and $B$ has chain
of $J$-sets,}{\Large\par}

{\Large{}
\[
A\supseteq A_{1}\supseteq A_{2}\supseteq\ldots\supseteq A_{n}\supseteq\ldots
\]
}{\Large\par}

{\Large{}
\[
B\supseteq B_{1}\supseteq B_{2}\supseteq\ldots\supseteq B_{n}\supseteq\ldots
\]
where each chain satisfies the conditions \ref{property 12.1} and
\ref{property 12.2}. }{\Large\par}

{\Large{}As, each $A_{i}\times B_{i}$ is $J$-set in $S_{1}\times S_{2}$
and 
\[
A\times B\supseteq A_{1}\times B_{1}\supseteq A_{2}\times B_{2}\supseteq\ldots\supseteq A_{n}\times B_{n}\supseteq\ldots
\]
and the rest of the proof is similar to the proof of theorem \ref{Theorem 2.16}
and so we leave it to the reader.}{\Large\par}
\end{proof}
{\Large{}Now we will give the proofs of theorem \ref{Theorem 3.24}
and \ref{Theorem 3.25} for non-commutative semigroup.}{\Large\par}

\subsubsection{\textbf{\Large{}Non-commutative case:}}
\begin{lem}
{\Large{}\label{Lemma 3.26} Let $\left(S,\cdot\right)$ be a semigroup
and $A\subseteq S$ is a J-set. Then for any $F\in\mathcal{P}_{f}$$\left(S^{\mathbb{N}}\right)$
and any IP-set $FS\left(\left\langle x_{n}\right\rangle _{n=1}^{\infty}\right)\subseteq\mathbb{N}$,
there exist $m\in\mathbb{N}$ and $\left(x_{i_{1}},x_{i_{2}},\ldots,x_{i_{m}}\right)\in\mathbb{N}^{m},$
$\left(a_{1},a_{2},\ldots,a_{m+1}\right)\in S^{m+1}$ such that 
\[
\left\{ a_{1}f\left(x_{i_{1}}\right)a_{2}f\left(x_{i_{2}}\right)\ldots a_{m}f\left(x_{i_{m}}\right)a_{m+1}:f\in F\right\} \subseteq A.
\]
}{\Large\par}
\end{lem}

\begin{proof}
{\Large{}Let, $G\in\mathcal{P}_{f}\left(S^{\mathbb{N}}\right)$ and
$G=\left\{ \left(f\left(x_{1}\right),f\left(x_{2}\right),\ldots,f\left(x_{n}\right),\ldots\right):f\in F\right\} $.
Then from definition \ref{definition 1.11} of J-set, the required
result follows.}{\Large\par}
\end{proof}
\begin{thm}
{\Large{}\label{Theorem 3.27} Let $\left(S_{1},\cdot\right)$ and
$\left(S_{2},\cdot\right)$ be two semigroup and $A\subseteq S_{1}$
and $B\subseteq S_{2}$ be two $J$-sets. Then $A\times B$ is a $J$-set
in $S_{1}\times S_{2}$.}{\Large\par}
\end{thm}

\begin{proof}
{\Large{}Let $A$ be a J-set and $F\in\mathcal{P}_{f}\left(S^{\mathbb{N}}\right)$.
Take any IP-set $FS\left(\left\langle x_{n}\right\rangle _{n=1}^{\infty}\right)$.
Now take the partition of $FS\left(\left\langle x_{n}\right\rangle _{n=1}^{\infty}\right)$
as,}{\Large\par}

{\Large{}$FS\left(\left\langle x_{n}\right\rangle _{n=1}^{\infty}\right)=B\cup C$,
where $\sum_{t\in G}x_{t}\in B$ if and only if there exists $a=\left(a_{1},a_{2},\ldots,a_{m+1}\right)\in S^{m+1}$
such that 
\[
\left\{ a_{1}f\left(x_{1}\right)a_{2}f\left(x_{2}\right)\ldots a_{m}f\left(x_{m}\right)a_{m+1}:f\in F\right\} \subseteq A
\]
otherwise $\sum_{t\in G}x_{t}\in C$ . Clearly one of $B$ or $C$
contain an $IP$-set from Hindman's Theorem \ref{Theiorem 3.22} but
$C$ can't contain this from Lemma \ref{Lemma 3.26} and construction
of $C$. So, $B$ will contain an IP-set $FS\left(\left\langle y_{n}\right\rangle _{n=1}^{\infty}\right)\subseteq FS\left(\left\langle x_{n}\right\rangle _{n=1}^{\infty}\right)$. }{\Large\par}

{\Large{}Now, let $G\in\mathcal{P}_{f}\left(\left(S\times S\right)^{\mathbb{N}}\right)$
and so,
\[
G=\left\{ \left(f_{1},g_{1}\right),\left(f_{2},g_{2}\right),\ldots,\left(f_{n},g_{n}\right):f_{i},g_{i}\in\mathcal{P}_{f}\left(S^{\mathbb{N}}\right)\,\forall i\in\left\{ 1,2,\ldots,n\right\} \right\} 
\]
 and therefore there exists an $IP$-set $FS\left(\left\langle y_{n}\right\rangle _{n=1}^{\infty}\right)$
such that for any $H\in\mathcal{P}_{f}\left(\mathbb{N}\right)$, $\left(y_{i_{1}},y_{i_{2}},\ldots,y_{i_{\left|H\right|}}\right)$,
there exists $a=\left(a_{1},\ldots,a_{\left|H\right|+1}\right)$ such
that 
\[
\left\{ a_{1}f\left(y_{i_{1}}\right)a_{2}f\left(y_{i_{2}}\right)\ldots a_{\left|H\right|}f\left(y_{i_{\left|H\right|}}\right)a_{\left|H\right|+1}:f\in\left\{ f_{1},f_{2},\ldots,f_{n}\right\} \right\} \subseteq A.
\]
}{\Large\par}

{\Large{}Now, using same argument as in the first part of the proof
, we have a sub IP-set $FS\left(\left\langle z_{n}\right\rangle _{n=1}^{\infty}\right)\subseteq FS\left(\left\langle y_{n}\right\rangle _{n=1}^{\infty}\right)$
such that for any $K\in\mathcal{P}_{f}\left(\mathbb{N}\right)$ and
$\left(z_{j_{1}},z_{j_{2}},\ldots,z_{j_{\left|K\right|}}\right)$,
there exists $b=\left(b_{1},b_{2},\ldots,b_{\left|K\right|+1}\right)$
such that 
\[
\left\{ \prod_{i=1}^{\left|K\right|}b_{i}g\left(z_{j_{i}}\right)b_{\left|K\right|+1}:g\in\left\{ g_{1},g_{2},\ldots,g_{n}\right\} \right\} \subseteq B.
\]
}{\Large\par}

{\Large{}So, 
\begin{align*}
\{\left(a_{1}f\left(z_{j_{1}}\right),b_{1}g\left(z_{j_{1}}\right)\right),\ldots,\left(a_{\left|K\right|}f\left(z_{j_{\left|K\right|}}\right),b_{\left|K\right|}g\left(z_{j_{\left|K\right|}}\right)\right)\cdot\left(a_{\left|K\right|+1},b_{\left|K\right|+1}\right) & :.\\
\left(f,g\right)\in G\subseteq A\times B\}
\end{align*}
}{\Large\par}

{\Large{}So, $A\times B$ is a J-set.}{\Large\par}
\end{proof}
\begin{thm}
{\Large{}\label{Theorem 3.28 } Let $\left(S_{1},\cdot\right)$ and
$\left(S_{2},\cdot\right)$ be two semigroup and $A\subseteq S_{1}$
and $B\subseteq S_{2}$ be two $C$-sets. Then $A\times B$ is a $C$-set
in $S_{1}\times S_{2}$.}{\Large\par}
\end{thm}

\begin{proof}
{\Large{}The proof is similar to the proof of theorem \ref{Theorem 3.25}
and so we leave it to the reader.}{\Large\par}
\end{proof}
\textbf{\Large{}Acknowledges:}{\Large{} The first author acknowledges
the grant UGC-NET SRF fellowship with id no. 421333 of CSIR-UGC NET
December 2016.}{\Large\par}

\end{document}